\documentclass[11pt, reqno]{amsart}
\usepackage{amsmath,amsthm,amssymb,bm}
\usepackage[utf8]{inputenc}  
\usepackage[T1]{fontenc} 
\usepackage{floatrow}
\usepackage{hyperref}
\usepackage{mathrsfs}
\usepackage{graphicx}
\usepackage{float}
\usepackage{algorithm}
\usepackage{algpseudocode}
\usepackage{tikz}
\usepackage{enumitem}
\usepackage{verbatim}
\usepackage{url}

\everymath{\displaystyle}

\definecolor{light-gray}{gray}{0.95}

\theoremstyle{plain}
\newtheorem{theorem}{Theorem}
\newtheorem{corollary}[theorem]{Corollary}
\newtheorem{lemma}[theorem]{Lemma}
\newtheorem{proposition}[theorem]{Proposition}
\theoremstyle{definition}

\theoremstyle{remark}
\newtheorem{remark}[theorem]{Remark}

\newcommand{\W}{\mathcal{W}}

\newcommand{\lebes}{\mathcal{L}^N}
\newcommand{\dd}{\mathrm{d}}

\newcommand{\supp}{\mathrm{supp}}
\newcommand{\Ob}{\overline{\Omega}}
\newcommand{\Om}{\Omega}

\newcommand{\K}{(\mathcal K)}
\newcommand{\F}{\mathcal F}

\newcommand{\M}{\mathcal M}
\newcommand{\D}{\mathcal D}

\newcommand{\uu}{\mathbf {u}}

\newcommand\restr[2]{{
		\left.\kern-\nulldelimiterspace 
		#1 
		\vphantom{\big|} 
		\right|_{#2} 
}}
\newcommand{\mres}{\mathbin{\vrule height 1.6ex depth 0pt width
		0.13ex\vrule height 0.13ex depth 0pt width 1.3ex}}

\def\1{\raisebox{2pt}{\rm{$\chi$}}}

\allowdisplaybreaks

\usepackage{mathtools}
\mathtoolsset{showonlyrefs}

\newcommand{\R}{\mathbb{R}}

\newcommand{\KR}{(\mathcal{KR})_H}
\newcommand{\KP}{(\mathcal{K})_H}

\newcommand{\B}{(\mathcal{B})}

\allowdisplaybreaks

\title[]{Quasi-convex Hamilton--Jacobi equations via limits of Finsler $p$-Laplace problems as $p\to \infty$}
\author[H.  Ennaji]{Hamza Ennaji\cs}
\author[N. Igbida]{Noureddine Igbida\cs}\thanks{\cs Institut de recherche XLIM, UMR-CNRS 7252, Facult\'e des Sciences et Techniques, Universit\'e de Limoges, France. \\Emails: hamza.ennaji@unilim.fr, noureddine.igbida@unilim.fr}
\author[V. T. Nguyen]{Van Thanh Nguyen\cm}\thanks{\cm Department of Mathematics and Statistics, Quy Nhon University, Vietnam.\\ Email: nguyenvanthanh@qnu.edu.vn}

\newcommand{\cs}{$^\dagger$} \newcommand{\cm}{$^\ddagger$}

\date{\today}

\begin{document}

	\begin{abstract}
		In this paper we show that the maximal viscosity solution of a class of quasi-convex Hamilton--Jacobi equations, coupled with inequality constraints on the boundary, can be recovered by taking the limit as $p\to\infty$ in a family of Finsler $p$-Laplace problems. The approach also enables us to provide an optimal solution to a Beckmann-type problem in general Finslerian setting and allows recovering a bench of known results based on the Evans--Gangbo technique.
	\end{abstract}
	
	\maketitle
	
	\section{Introduction}	
	Let $\Om$ be a smooth bounded subset of $\R^N$. Consider a continuous Hamiltonian $F:\Ob\times\R^N\rightarrow\R$ such that, for all $x\in \overline\Omega$, 
	\begin{itemize}
		\item $Z(x):=\{\xi \in \R^N:  ~F(x,\xi)\leq 0\}$ is a convex and compact subset of $\R^N$.
		\item $0\in \hbox{int} (Z(x)).$ 
	\end{itemize}
Our main aim concerns the Hamilton--Jacobi (HJ for short) equation of first order
	\begin{equation} \label{hj0}
	F(x,\nabla u)=0  \hbox{ in }\Omega.
	\end{equation}
 
	\bigskip 
	The class of HJ PDE  is  central  in several branches of mathematics, both from theoretical, numerical and application points of view. The applications in classical mechanics, optics, Hamiltonian dynamics, semi-classical quantum theory, Riemannian and Finsler geometry as well as the optimal control theory are very important.

In addition to its connection with Hamilton’s equations, in the case where the Hamiltonian has sufficient regularity, further  connection with common PDEs was established in the literature. For instance, it appears in the classical limit of the Schrödinger equation (see e.g. \cite{bardos}). Its connection with the discount HJ equation $\lambda u + 	F(x,\nabla u)=0 $ as $\lambda \to 0$ was established in the  seminal paper  \cite{LiPaVa} and generalized in \cite{DaFaIt}. The vanishing viscosity method for first order HJ equations establishes  the connection of HJ equations with the second order PDE   $-\epsilon \Delta u + 	F(x,\nabla u)=0$ as $\epsilon \to 0$ (see for instance \cite{CL,L2}). The celebrated paper of Varadhan \cite{Varadhan} shows that the heat kernel in a Riemannian manifold can be approximated by a Gaussian kernel, and thus makes the link between the heat equation and the HJ equation. This connection can be also done via Hopf--Cole transformations as showed in \cite{IDC}. This kind of transformations also allows recovering the HJ equation in the large scale hyperbolic limit of a class of kinetic equation (see e.g. \cite{Boun&Calvez}).

	Recently, the connection between HJ equation, optimal mass transport and Beckmann's problem was established in  \cite{ennaji2020augmented,HINBeckHJ} with a flavor of variational approach.  In particular, these connections work out a nonlinear divergence-form PDE, called Monge--Kantorovich equation,  that we can associate definitively with the HJ equation. The connection is not straightforward since the optimal mass transportation, the Beckmann's problem as well as the associate divergence formulation are not standard. Roughly speaking,  the offset is connected to some unknown distribution of mass concentrated on the boundary which would both, counterbalance  the involved optimal mass transportation phenomena and describe the normal-trace of the allowed   flux in the divergence formulation (see  \cite{ennaji2020augmented,HINBeckHJ} for the details).  The approach   blends sophisticated tools from variational analysis, convex duality and  trace-like operator  for the so called divergence-measure field.  To strengthen the connection with divergence equation and to shape the ''pretending diffusive taste'' of HJ equation,  we propose in this paper how to achieve  the solutions of HJ equation using an elliptic PDE of Finsler $p-$Laplace type. The Finsler structure associated with  the Hamiltonian $F$ takes part in  the PDE in a common way bringing out  some kind of anisotropic $p-$Laplace PDE, that we call here  Finsler $p-$Laplace equation.   We treat the equation \eqref{hj0}  with a double obstacle  on the boundary.  Moreover, thanks to the substantial link of HJ equation with the optimal mass transport as well as the Beckmann problem, these problems  will be concerned  in their turn with the approach using the  Finsler $p-$Laplace equation.
	  
\bigskip 

To describe roughly the approach, we consider the peculiar case  of eikonal equation with Dirichlet  boundary condition: 
\begin{equation}
	\label{eik0} \left\{ \begin{array}{ll}
		\vert \nabla u\vert =k \quad &\hbox{ in }\Omega\\  \\ 
	u=g& \hbox{ on } \partial \Omega,		
	\end{array}		
	\right.
\end{equation}
where $k$ is a positive continuous function in $\overline \Omega$ and   $\partial \Omega$ denotes the boundary of $\Omega$.  
	It is well known  by now that   the intrinsic distance  defined by 
	\begin{equation}
		\label{distanceF}
		d_{k}(x,y) := \inf_{\zeta\in \Gamma(x,y)} \int_{0}^{1}k( \zeta(t))\: \vert \dot{\zeta}(t)\vert  \dd t,
	\end{equation}
	where $\Gamma(x,y)$ is the set of Lipchitz curves joining $x$ and $y$,  describes  the maximal 
	  viscosity subsolution  through the following formula 
	\begin{equation}
		\label{viscSol}
		u(x) = \min_{y\in \partial\Om  } \left\{d_{k}(y,x) + g(y)\right\}.
	\end{equation}
 	Here  $g:\partial\Om\rightarrow\R$   is assumed  to be a continuous function satisfying the compatibility condition 
 \begin{equation}
 	\label{compat_g}
 	g(x) - g(y) \leq d_{k}(y,x),~~\mbox{for all}~x,y\in\partial\Om.
 \end{equation}  
Since  \eqref{viscSol} is likewise the unique solution of  the following maximization problem
	\begin{equation}
		\label{maxv0}
		\max_{z\in W^{1,\infty}(\Om)}\Big\{ \int_{\Omega} z(x)\dd x: ~ \vert \nabla z(x)\vert  \leq k(x)~\mbox{and}~z=g~\mbox{on}~\partial\Om \Big\}, 
	\end{equation}
we know (see  \cite{ennaji2020augmented,HINBeckHJ}) that a  dual problem of \eqref{maxv0} reads
	\begin{equation}\label{Bek0}
		\min_{\phi\in\mathcal{M}_{b}(\Ob)^{N},\: \nu\in\mathcal{M}_b(\partial\Om)}\left\{   \int_{\overline\Omega} k\: \dd\vert\phi\vert  + \int_{\partial\Om} g \dd\nu:~~-\operatorname{div}(\phi) = \chi_{\Omega}  -\nu~\mbox{in}~~\mathcal{D}^{'}(\R^N)\right\},
	\end{equation}
	which constitute  actually a new variant of Beckmann's problem with boundary cost $g$.  Here $\mathcal{M}_{b}  $ is used to denote the set of finite Radon measures. In particular,  this is  connected to the Monge optimal mass transport problem 
		\begin{equation}
		\label{M}
		 	\inf\left\{   \int_{\Omega}\: d_k( x , T(x)) \dd x \: :\:   \nu\in\mathcal{M}_b(\partial\Om),\: T_{\sharp}\chi_{\Omega}  = \nu  \right\} 
	\end{equation} 
	as well as to the Monge--Kantorovich relaxed  problem  
		\begin{equation}
		\label{K}
		\min \left \{\int_{\Omega\times\Omega} d_k( x , y ) \dd\gamma(x,y)\: :\:       \nu\in\mathcal{M}_b(\partial\Om),\:     \gamma\in \mathcal{M}^+(\Omega\times \Omega),\:  (\pi_{x})_{\sharp}\gamma= \chi_\Omega , (\pi_{y})_{\sharp}\gamma= \nu \ \right \}.
	\end{equation}
Even if  here the so called target measure $\nu$ is an unknown parameter of the problem, one sees that the problem aims certainly an optimal mass transportation between $\rho_1:= \chi_{\Omega} $ and $\rho_2:=\nu$,  and moreover $u,$ given by \eqref{viscSol} (the unique solution of \eqref{eik0}) is an Kantorovich potential of transportation.    Since the pioneering work of  Evans-Gangbo (cf. \cite{evans1999differential}) in the case where $k\equiv 1$, it is known that key information  concerning $u$ may be  given by 
the uniform limit of $u_p$, the  solution of the modified $p-$Laplace equation 
\begin{equation} 		\label{plapeik}
 	\left\{\begin{array}{ll} 	-\Delta_p \left( \frac{u_p}{k} \right) = \rho_1-\rho_2\quad  & \text {in } \overline  \Omega \\  
 		u	=g  & \text { on } \partial \Omega .  \end{array}\right.  
\end{equation}
Following the results of \cite{evans1999differential}, one can guess this  limit to be  given by the so-called Monge--Kantorovich system: 
	\begin{equation}
		\label{mksys}
		\left\{\begin{array}{ll}
			-\operatorname{div}(\Phi )=\rho_1-\rho_2, \:  	\vert \nabla u\vert \leq k    \quad & \text {in } \overline \Omega \\     
	 \Phi=  m \: \nabla u  ,\:	m\geq 0,\:  m(	\vert\nabla u\vert -k)=0   & \text { a.e. }\\  
	u	=g  & \text { on } \partial \Omega.   \end{array}\right.
	\end{equation}
Notice here  that,  a part a few  special cases out of the scope of our situation  (cf. \cite{santambrogio2015optimal} Chap. 4.3 for discussions and references about regularity properties  of $\Phi$ under extra assumptions),  in general   the flux $\Phi $ is a     vector valued  measure, and it is closely   connected to  the  solution of  Beckmann problem \eqref{Bek0}.   Coming back to the HJ equation \eqref{eik0}, it is clear now that the Monge--Kantorovich system is a suitable divergence equation for the solution of \eqref{eik0}. Moreover, the  limit of the  flux of \eqref{plapeik} converges weakly to $\Phi$   picturing thereby   some kind of ''nonlinear diffusion'' phenomena behind the Hamilton-Jacobi equation. 

\subsection*{Contributions} 
In this paper, we are interested in studying the connection between the HJ equation, coupled with inequality constraints on the boundary, 
\begin{equation}
	\label{hjobs0} \left\{ \begin{array}{ll}
		F(x,\nabla u)=0 \quad &\hbox{ in }\Omega\\  \\ 
		\phi\leq u\leq \psi & \hbox{ on } \partial \Omega	
	\end{array}		
	\right.
\end{equation}
and an elliptic problem of Finsler $p-$Laplace type that we will introduce below.  

We show  how  to recover the maximal viscosity subsolution to the class of HJ  equations of the type \eqref{hjobs0}
using a family of Finsler $p$-Laplace problems (with boundary obstacles) as $p\to\infty$.  Moreover, since the solution of \eqref{hjobs0}   is intimately linked to the so called Kantorovich-Rubinstein problem in optimal transport, an appropriate Beckmann's transportation problem is derived and its solution is provided. Essentially, this will be the content of Theorem \ref{thm:main_results} whose proof relies on the results and estimates of Propositions \ref{psolfinsp} and \ref{prop:main_estimates}.  Finally, we show in Proposition \ref{prop4} that the limit as $p\to\infty$ of solutions of the $p$-Laplace problems is a Kantorovich potential for a classical Kantorovich problem involving the normal trace on the boundary of the optimal flow of Beckmann's problem. Our work illustrates some kind of ”nonlinear diffusion” phenomena behind the Hamilton-Jacobi equation.

	\subsection*{Related works}
	Concerning limits as $p\to \infty$ for the $p$-Laplace equations, one of the first mathematical studies is  \cite{bhattacharya1989limits} with particular interest in torsional problems and $\infty$-harmonic functions, followed by the celebrated work of Evans and Gangbo \cite{evans1999differential}. Similar problems were considered in \cite{Rossi1,Rossi2} for transport problems with masses supported on the boundary. Variants of Monge-Kantorovich problems  with boundary costs were addressed in \cite{mazon2014optimal} where the boundary costs can be seen as some import/export taxes. In the same spirit, similar results were obtained in \cite{dweik2018weighted} with some weighted Euclidean distance as a cost. The use of PDE techniques à la Evans--Gangbo in the Finsler framework was addressed recently in \cite{igbida2018optimalmass}. It is well known that Finsler metrics generalise the Riemannian ones and are of main interest in the study of optimal transport and minimal flow problems since they allow considering anisotropy, obstacles...
	
	Our work adds to these series of papers linking HJ equations to other PDE's, thanks to the variational approach (cf. \cite{ennaji2020augmented}) and permits generalizing the works on mass transport recalled above. It shows once again the flexibility of the Evans-Gangbo method.

\bigskip 

 The rest of this paper is organized as follows. In section \ref{pre}, we present assumptions and preliminary results concerning the notion of solution to the HJ equation coupled with obstacles on the boundary under consideration, Finsler $p$-Laplace equations as well as their existence and characterization of solutions. In section \ref{limits}, we derive suitable estimates independent of $p$ and show the convergence of Finsler $p$-Laplace equations as $p\to \infty$. The existence and characterization of solutions to the limited variational problems are also studied in detail. Finally, the connection between the limited variational problems and a variant of Monge--Kantorovich transportation problem is derived in section \ref{connection}.

  \section{Preliminaries}\label{pre}
  \subsection{Maximal viscosity subsolution}
  	\hfill\\
  Consider the Hamilton--Jacobi equation of first order,  coupled with some inequality constraints on the boundary
  \begin{equation}\label{hjobs}
  \begin{cases}
 F(x, \nabla u)=0 & \text{ in } \Omega\\
 \phi\le u\le \psi & \text{ on }\partial \Omega.
  \end{cases}
  \end{equation}
  Here, $\phi,\psi \in C(\partial\Omega)$ satisfy the compatibility condition
  \begin{equation}
  \label{comp000}
  \phi(x) - \psi(y) \leq d_{\sigma}(y,x)~\hbox{for all}~x,y\in\partial\Omega, 
  \end{equation} 
  with $d_\sigma$ being the intrinsic metric associated to $F$ (see below).

For each $x\in\Ob$, we define the support function $\sigma(x, .)$ of the $0$-sublevel set of $F$ by
	\begin{equation}
		\sigma(x,q) = \sup_{p\in Z(x)} \langle p,q\rangle ~~~\mbox{for all}~q\in\R^N,
	\end{equation}
	which turns to be a Finsler metric (see subsection \ref{subsec:FinslerpLaplace} below). Then, the intrinsic distance associated to $F$ is defined through
	\begin{equation}
		\label{distanceF}
		d_{\sigma}(x,y) := \inf_{\zeta\in \Gamma(x,y)} \int_{0}^{1} \sigma(\zeta(t),\dot{\zeta}(t)) \dd t,
	\end{equation}
	where $\Gamma(x,y)$ is the set of Lipchitz curves joining $x$ and $y$. In the case where $\phi \equiv \psi = g:\partial\Om\rightarrow\R$ is a continuous function satisfying the compatibility condition 
	\begin{equation}
	\label{compat_g_sigma}
	g(x) - g(y) \leq d_{\sigma}(y,x)~~\mbox{for all}~x,y\in\partial\Om,
	\end{equation}
	it is well known (see e.g. \cite{Fathi,L2}) that the maximal viscosity subsolution of
	\begin{equation}\label{hj1}
		\begin{cases}
		F(x,\nabla u)=0 & \text{ in } \Omega\\
		u=g &\text{ on } \partial\Om
		\end{cases}
	\end{equation} 
	is given by
	\begin{equation}
		\label{viscSol1}
		u(x) = \min_{y\in \partial\Om  } \left\{d_{\sigma}(y,x) + g(y)\right\}.
	\end{equation} 
Moreover, this solution coincides with the maximal volume solution.  Indeed, using the fact that the set of all viscosity subsolutions of \eqref{hj1} coincides with the set of Lipschitz functions $u$ satisfying
\begin{equation}
	\label{sigma_star_ball}
	\sigma^{*}(x,\nabla u(x))\leq 1~\mbox{a.e.},
\end{equation}
where $\sigma^*$ is the dual of the support function $\sigma$ defined through
\begin{equation}
	\sigma^{*}(x,q) = \sup_{\sigma(x,p)\leq 1} \langle p,q \rangle,
\end{equation}
we proved in \cite{ennaji2020augmented} that \eqref{viscSol1} is the unique solution of the following maximization problem
\begin{equation}
	\label{maxv}
	\max_{z\in W^{1,\infty}(\Om)}\Big\{ \int_{\Omega} z(x)\dd x,~\sigma^{*}(x,\nabla z (x)) \leq 1~\mbox{and}~z=g~\mbox{on}~\partial\Om \Big\}.
\end{equation}

Now, for the study of the general problem \eqref{hjobs} with inequality constraints on the boundary, we make use of a similar notion  of solution. 
Actually we have 
 \begin{proposition}
 	Under the assumption \eqref{comp}, the problem \eqref{hjobs} has a unique solution $u$ in the sense of maximal volume, that is, $u$ is the unique solution to the following maximization problem
 	\begin{equation}
 		\label{maxobs}
 		\max_{z\in W^{1,\infty}(\Om)}\Big\{ \int_{\Omega} z(x)\dd x,~\sigma^{*}(x,\nabla z (x)) \leq 1~\mbox{and}~\phi\leq z\leq \psi ~\mbox{on}~\partial\Om \Big\}.
 	\end{equation} 
 	Moreover, $u$ is the maximal viscosity subsolution satisfying  $\phi\leq u\leq \psi$ on $ \partial\Om.$ 
 \end{proposition}

 	\subsection{Finsler $p-$Laplacian equation}\label{subsec:FinslerpLaplace}
 	\hfill\\
	Let $\Omega$ be a bounded open subset of $\R^N$, a Finsler metric is a continuous function $H: \overline\Om \times \R^N \to [0,\infty)$ such that $H(x, .)$ is convex, and positively $1-$homogeneous in the second variable, that is, $H(x,tp) =t H(x,p)$  for every $t\geq 0$. 
	\\We define the dual of a Finsler metric $H$  (which is also a Finsler metric) by
	\[
	H^*(x,q) = \sup_{H(x,p)\leq 1} \langle p,q\rangle
	=\sup_{p\neq 0} \frac{\langle p,q\rangle}{H(x,p)}.\]
	
	In this paper, we assume that $H$ is a non-degenerate Finsler metric, that is, there exist $a,b >0$ such that
	\begin{equation}
		\label{equiv}
		a\vert p\vert \leq H(x,p)\leq b \vert p\vert
	\end{equation}
	for all $(x,p)\in \overline\Omega\times \R^N$. In other words, one has 
	\begin{equation}
		\label{equivhs}
	\tilde{a}\vert q\vert \leq H^*(x, q)\leq \tilde{b}\vert q\vert
	\end{equation}
	for some $\tilde{a},\tilde{b}>0$. Moreover, we have the Cauchy--Schwarz like inequality
	
	\begin{equation}
		\label{cs}
		\langle p,q\rangle \leq H(x,p) H^* (x,q).
	\end{equation}
	Euler's homogeneous function theorem (see e.g. \cite{nesterov2013introductory}) says that 
	\begin{equation}
		\label{ehf}
		\partial_\xi H^{*}(x,p)\cdot p=  H ^*(x,p)~~\mbox{for any}~~p\in\R^N,
	\end{equation}
	and by convexity of $H^*$, we have
	\begin{equation}
		\label{fin_p1}
		\partial_\xi H ^{*}(x,p)\cdot q\leq  H ^*(x,q)~~\mbox{for any}~~p,q\in\R^N.
	\end{equation}
	Thus, using \eqref{equivhs} we get
	\begin{equation}
		\label{fin_p2}
		\vert\partial_\xi H ^{*}(x,p)\cdot q\vert \leq  \tilde{b}\vert q\vert~~\mbox{for any}~~p,q\in\R^N.
	\end{equation}
	Finally, we have
	\begin{equation}
		\label{fin_p3}
		H(x,\partial_\xi H^*(x,p)) = 1 ~~\mbox{for any}~p\in\R^N.
	\end{equation}
For details and additional properties we refer the reader to \cite{Schneiuder}.

\bigskip 
	Every Finsler metric induces a Finsler distance via the so called length (or action) functional. The action of a Lipschitz curve $\xi\in \hbox{Lip}([0,1];\Ob)$ is defined through
	\begin{equation}\label{action}
		A_{H}(\xi) = \int_{0}^{1} H(\xi(s),\dot\xi(s)) \dd s.
	\end{equation}
	The induced distance $d_H$ by the action functional \eqref{action} reads as
	\begin{equation}
		\label{dfin}
		d_H(x,y) = \inf_{\xi\in \Gamma(x,y)} A_H(\xi).
	\end{equation}
Note that in general, $H(x,p)$ is not even in $p$ so that $d_H$ may be non-symmetric, i.e., it may happen that $d_{H}(x,y) \neq d_{H}(y,x)$. 

\bigskip 
Assuming that $H ^*(x, .)\in C^1(\R^N\setminus\{0\})$ and the compatibility condition 
	\begin{equation}\label{comp}
	\phi(x) - \psi(y) \le d_H(y,x)~\hbox{for all}~x,y\in\partial\Omega,
	\end{equation}
	we consider the following Finsler (also called anisotropic) $p-$Laplace  problems

	\begin{equation}
		\label{MKp0}
		\left\{\begin{array}{ll}-\operatorname{div}(H ^{*}(x,\nabla u_p)^{p-1}\partial_{\xi}H^{*}(x,\nabla u_p))=\rho \quad & \text {in } \Omega \\   \\  
			\phi \leq u_p \leq \psi & \text {on } \partial \Omega,\\ 
		\end{array}\right.
	\end{equation}	
	where  $p>N$ and $\rho\in L ^2(\Omega)$ are given, and $\partial_\xi H^*$ stands for the derivative of $H^*$ with respect to the second variable.   To study this problem let us consider the set 
		$$\W_{\phi,\psi} = \{ u\in W^{1,p}(\Omega):~\phi \leq u \leq\psi ~\mbox{on}~\partial\Om\}$$
		and we denote by 
	\begin{equation}
		\Theta_p = H ^{*}(x,\nabla u_p) ^{p-1} \partial_{\xi}H ^{*}(x,\nabla u_p).
	\end{equation}  
	\begin{proposition}
			Assume \eqref{comp} is strict, that is,
		\begin{equation}
			\label{scomp00}
			\phi(x) - \psi(y) < d_H(y,x)~\hbox{for all}~x,y\in\partial\Omega.
		\end{equation}
		\label{psolfinsp}
		The problem  \eqref{MKp0} has a unique solution $u_p  $ in the following sense: $u_p\in \W_{\phi,\psi}  $ and 
		\begin{equation}\label{varsolp}
			\int_{\Omega  }  	\Theta_p \cdot \nabla (u_p-\xi)\: \dd x \leq \int_\Omega \rho\: (u_p-\xi)\: \dd x \quad \hbox{ for any }\xi\in   \W_{\phi,\psi}. 
		\end{equation}
		Moreover, the distribution defined through
		\begin{equation}
			\label{dist}
			\langle \Theta_p\cdot\mathbf{n},\eta\rangle = \int_{\Omega}\Theta_p\nabla\eta \dd x- \int_{\Omega}\eta\rho \dd x,~\eta\in\D(\R^N),
		\end{equation}
		is a Radon measure concentrated on $\partial\Omega$ which satisfies 
		\begin{equation}
		\label{mf1}
		\int_{\Om} \Theta_p\cdot\nabla \eta\dd x = \int_{\Om} \eta\rho\dd x+ 	\int_{\partial\Om}\eta \dd(\Theta_p\cdot\mathbf{n})~~\mbox{for all}~\eta\in W^{1,p}(\Om),
	\end{equation}
	 and 	\begin{equation}  \label{boundarybehaviorp}
	 	\supp(\left(\Theta_p\cdot\mathbf{n}\right)^{+})\subset  \{u_p = \phi\}~~\mbox{and}~~\supp(\left(\Theta_p\cdot\mathbf{n}\right)^{-})\subset\{u_p = \psi\}.
	 \end{equation} 
	\end{proposition}
	\begin{proof} We consider   the following minimization problem of Finsler $p$-Laplace type
	\begin{equation}
		\label{vp}
		\min_{u\in \W_{\phi,\psi}}  \F_p(u):=\int_{\Omega} \frac{H ^*(x,\nabla u)^p}{p} \dd x - \int_{\Omega} u \rho \dd x. 
	\end{equation}
   Observe that $\W_{\phi,\psi}$ is a closed, convex subset of $W^{1,p}(\Omega)$. The functional  $\F_p$ is coercive, strictly convex and lower semicontinuous on $\W_{\phi,\psi}$. Therefore $\F_p$ admits a unique minimizer on $\W_{\phi,\psi}$ which satisfies \eqref{varsolp}. 
    
		Now, to prove \eqref{mf1} we follow the main ideas of \cite[Thereom 3.4]{mazon2014optimal}.  Clearly,  \eqref{varsolp}  implies $-\operatorname{div}(\Theta_p) = \rho$ in $\mathcal{D}^{'}(\Om)$. It follows that the  $\Theta_p\cdot\mathbf{n}$ defined by \eqref{dist} is a distribution supported on $\partial\Omega$. Let us show moreover that 
		\begin{equation} \supp\left(\Theta_p\cdot\mathbf{n}\right) \subset \left\{x\in\partial\Om:~u_p(x) = \phi(x)\right\}\cup \left\{x\in\partial\Om:~u_p(x) = \psi(x)\right\}.
		\end{equation}
		Take a test function $\eta\in C^{\infty}(\Ob)$ whose support is disjoint from $\left\{x\in\partial\Om:~u_p(x) = \phi(x)\right\}\cup \left\{x\in\partial\Om:~u_p(x) = \psi(x)\right\}$. There exists some $\epsilon>0$ so that $u_p+t\eta$ remains admissible for \eqref{vp} for $|t|<\epsilon$, \textit{i.e.}, $\phi \leq u_p + t\eta \leq \psi$. By optimality of $u_p$, we get the variational inequality
		\[
		\int_{\Om} \Theta_p\cdot\nabla (v - u_p) \dd x \geq \int_{\Om} (v - u_p) \rho \dd x~~\mbox{for all}~v\in \W_{\phi,\psi}.
		\]
		In particular, for $v =  u_p + t\eta$, we get
		\[
		t \int_{\Om} \Theta_p\cdot \nabla \eta \dd x \geq t \int_{\Om} \eta\rho\dd x.
		\]
		This holds for positive and negative $t$, such that $\vert t\vert\leq\epsilon$. Consequently
		\[
		\int_{\Omega} \Theta_p\cdot \nabla \eta \dd x = \int_\Omega \eta \rho \dd x .
		\]
		In other words, $\langle \Theta_p\cdot\mathbf{n}, \eta\rangle = 0$ and $\supp(\Theta_p\cdot\mathbf{n}) \subset \{u_p = \phi\}\cup \{u_p= \psi\}.$ We are now in a position to show that $\Theta_p\cdot\mathbf{n}$ is actually a Radon measure. Indeed, the inequiality \eqref{scomp00} implies that the two compact sets $\left\{x\in\partial\Om:~u_p(x) = \phi(x)\right\}$ and $\left\{x\in\partial\Om:~u_p(x) = \psi(x)\right\}$ are disjoint. There exist $\eta_1,\eta_2\in\D(\R^N)$ such that
		\[
		\eta_1(x) = \begin{cases}
			1~\mbox{on}~\{u_p = \phi\},\\0~\mbox{on}~\{u_p= \psi\},
		\end{cases}~~\mbox{and}~~		\eta_2(x) = \begin{cases}
			1~\mbox{on}~\{u_p= \psi\},\\0~\mbox{on}~\{u_p = \phi\}.
		\end{cases}
		\]
		Then we can write $\Theta_p\cdot\mathbf{n}=D_1+D_2$, where $D_1, D_2$ are distributions given by
		\[
		\langle D_1,\eta\rangle = \langle \Theta_p\cdot\mathbf{n}, \eta\eta_1 \rangle~~\mbox{and}~~\langle D_2,\eta\rangle = \langle \Theta_p\cdot\mathbf{n}, \eta\eta_2 \rangle.
		\]
		This being said, for any positive test function $\eta$, we have that $\supp(\eta\eta_1)\cap \{u_p= \psi\} = \emptyset$, and  for $0\leq t<\epsilon$ we have $u_p +t (\eta\eta_1)\in \W_{\phi,\psi}$. Consequently
		\[
		t \int_{\Om} \Theta_p\cdot \nabla (\eta\eta_1) \dd x \geq t \int_{\Om} (\eta\eta_1)\rho\dd x,
		\]
		\textit{i.e}, 
		\begin{equation}
			\label{dist1}
			\langle D_1,\eta\rangle \geq 0.
		\end{equation}
		On the other hand, for any positive test function $\eta$, we have that $\supp(\eta\eta_2)\cap \{u_p= \phi\} = \emptyset$ and for $-\epsilon < t\leq 0$, we have that  $u_p +t (\eta\eta_2)\in \W_{\phi,\psi}$. Consequently
		\[
		t \int_{\Om} \Theta_p\cdot \nabla (\eta\eta_2) \dd x \geq t \int_{\Om} (\eta\eta_2)\rho\dd x.
		\]
		In other words,
		\begin{equation}
			\label{dist2}
			\langle D_2,\eta\rangle \leq 0.
		\end{equation}
		In conclusion, $D_1$ and $-D_2$ are positive distributions. Hence, they are positive Radon measures. It follows that the distribution $\Theta_p\cdot\mathbf{n}$ is a Radon measure on $\partial\Om$. Moreover,  \eqref{dist1} and \eqref{dist2} give \eqref{boundarybehaviorp}.  
	
		\end{proof}

	Thanks to the proof of Proposition \ref{psolfinsp}, we have the following description of the solution. 
	
	\begin{corollary}
	If  $H ^*(x, .)\in C^1(\R^N\setminus\{0\})$, then $u_p$ is the unique solution of the problem
	  	
		\begin{equation}
			\label{MKp}
			\left\{\begin{array}{ll}-\operatorname{div}(H ^{*}(x,\nabla u_p)^{p-1}\partial_{\xi}H^{*}(x,\nabla u_p))=\rho \quad & \text {in } \Omega \\ 
				H ^{*}(x,\nabla u_p)^{p-1}\partial_{\xi}H ^{*}(x,\nabla u_p)\cdot\mathbf{n} \geq 0 & \text {on }\left\{u_p = \phi\right\} \\
				H ^{*}(x,\nabla u_p)^{p-1}\partial_{\xi}H ^{*}(x,\nabla u_p)\cdot\mathbf{n} \leq 0 & \text {on }\left\{u_p= \psi\right\} \\ 
				H ^{*}(x,\nabla u_p)^{p-1}\partial_{\xi}H ^{*}(x,\nabla u_p)\cdot\mathbf{n}  = 0 & \text {in } \{\phi < u_p <\psi\}\\
				\phi \leq u_p \leq \psi & \text {on } \partial \Omega, \\ 
			\end{array}\right.
		\end{equation}
		where $\mathbf{n}$ is the exterior normal to the boundary $\partial\Omega$, in the sense that  $u_p\in \W_{\phi,\psi}  $, $\Theta_p\in L^{p'}(\Omega)^N,$  $  \Theta_p\cdot\mathbf{n} \in \mathcal M_b(\partial \Omega),$   and the triplet $(u_p,\Theta_p, \Theta_p\cdot\mathbf{n} )$ satisfies  	\eqref{mf1}-\eqref{boundarybehaviorp}. 
		 
			\end{corollary}
	\begin{remark}
	In order to simplify the presentation we have assumed that $H^{*}(x,.)\in C^{1}(\R^n\setminus\{0\})$. However, we do believe that all the results of this paper remain true without this assumption and one needs just to replace the derivative of $H^*$ with respect to the second variable by the subdifferential.
	\end{remark}

	\section{Limits of Finsler $p$-Laplacian as $p\to\infty$}\label{limits}
The strategy is to obtain some uniform bounds in $p$ of $\nabla u_p$, then we show that the triplet $(u_p,\Theta_p, \Theta_p\cdot\mathbf{n} )$  converges (up to a subsequence) to optimal solutions of the corresponding Kantorovich-Rubinstein and Beckmann-type  problems. The following result gathers main estimates, that we will need later.

	\begin{proposition}[Main estimates]
		\label{prop:main_estimates}
	Assume \eqref{comp} is strict, that is,
	\begin{equation}
		\label{scomp0}
		\phi(x) - \psi(y) < d_H(y,x)~\hbox{for all}~x,y\in\partial\Omega.
	\end{equation}
	Then, we have
		\begin{enumerate}[label=(\roman*)]
			\item \label{item1} estimate on $u_p$
			\begin{equation}\label{estimate_up}
				\vert u_p(x) - u_p(y)\vert \leq C \vert x-y\vert^{r},~~\mbox{for all}~~x,y\in\Om; 
			\end{equation}
			
			\item \label{item2} estimates on $\Theta_p\cdot\mathbf{n}$:
			\begin{equation}\label{estimates_Thetapn}
				\int_{\partial\Omega} \dd(\Theta_p\cdot\mathbf{n})^{+} \leq C_1,~~\mbox{and}~~	\int_{\partial\Omega} \dd(\Theta_p\cdot\mathbf{n})^{-} \leq C_2;
			\end{equation}
			\item \label{item3} estimate on $\Theta_p$:
			\begin{equation}\label{estimate_Thetap}
				\int_\Omega \vert \Theta_p\vert \dd x \leq C,
			\end{equation}
		\end{enumerate}
		where $r, C, C_1, C_2$ are positive constants independent from $p$.
	\end{proposition}	
	\begin{proof} First, we prove \ref{item1}. Define $v(x) = \min_{y\in\partial\Omega} \psi(y) + d_H (y, x)$. Regarding the compatibility condition \eqref {comp}, we have $\phi \leq v \leq \psi $ on $\partial\Omega$. It is not difficult to see that $v$ is $1-$Lipschitz with respect to $d_H $ and  equivalently (see e.g. \cite[Proposition 2.1]{ennaji2020augmented}), we have that $H ^*(x,\nabla v(x)) \leq 1$  a.e. in $\Omega.$  Using the fact that $u_p$ is a minimizer of $\F_p$, we have
		\begin{equation}
			\label{ee1}
			\int_{\Omega} \frac{H ^*(x,\nabla u_p)^p}{p} \dd x - \int_{\Omega} u_p \rho \dd x \leq \int_{\Omega} \frac{H ^*(x,\nabla v) ^p }{p} \dd x - \int_{\Omega} v \rho \dd x\leq \frac{\vert \Omega\vert }{p} - \int_{\Omega} v \rho \dd x.
		\end{equation}
		Thanks to Theorem 2.E in \cite{talenti1994inequalities}, there is a Morrey-type inequality independent of $p$
		\begin{equation}
			\|u\|_{L^\infty(\Omega)}\le C_\Omega \|\nabla u\|_{L^p(\Omega)} \text{ for any } u\in W^{1,p}_0(\Omega),~p>N+1,
		\end{equation}
		where the constant $C_\Omega$ does not depend on $p$ and $u$. 
		Observing that we can apply the above inequality to $(u_p - \max_{\partial\Omega} \psi)^{+}$ and $(u_p - \min_{\partial\Omega}\phi)^{-}$ which are in $W^{1,p}_{0}(\Omega)$ to obtain
		\[
		\Vert u_p^+\Vert_{L^{\infty}(\Omega)} \leq C_\Omega\Vert\nabla u_p\Vert_{L^p(\Omega)} + \vert \max_{\partial\Omega}\psi\vert,
		\]
		and
		\[
		\Vert u_p^-\Vert_{L^{\infty}(\Omega)} \leq C_\Omega\Vert\nabla u_p\Vert_{L^p(\Omega)} + \vert \min_{\partial\Omega}\phi\vert.
		\]
		So $$\Vert u_p\Vert_{L^{\infty}(\Omega)} \leq C_1 \Vert\nabla u_p\Vert_{L^p(\Omega)} + C_2.$$ From \eqref{ee1} and the preceding inequality we deduce that
		\[
		\int_{\Omega} \frac{H ^*(x,\nabla u_p) ^p }{p} \dd x  \leq \frac{\vert \Omega\vert }{p} - \int_{\Omega} v \rho \dd x + \int_{\Omega} u_p \rho \dd x \leq C_3(1 + \Vert\nabla u_p\Vert_{L^p(\Omega)}),
		\]
		where $C_3$ is a positive constant not depending on $p$. Combining this with  \eqref{equiv}, we get 
		\begin{equation}\label{est00}
			\Vert H ^* (x,\nabla u_p)\Vert_{L^p(\Omega)}^{p}\leq  C_4p (1 +\Vert H ^* (x,\nabla u_p)\Vert_{L^p(\Omega)})
		\end{equation}
		which implies that 
		\begin{equation}
			\label{est1}
			\Vert H ^* (x,\nabla u_p)\Vert_{L^p(\Omega)}\leq (C_5p)^{\frac{1}{p-1}}
		\end{equation}
		for some constant $C_5$ independent from $p$. Again, by \eqref{equiv}, we get
		\begin{equation}
			\label{est2}
			\Vert \nabla u_p\Vert_{L^p(\Omega)} \leq C_6.
		\end{equation}
		Now take some $N<m\leq p$. Then by Hölder's inequality
		\begin{equation}
			\label{est3}
			\Vert \nabla u_p\Vert_{L^{m}(\Omega)} \leq \vert\Omega\vert^{\frac{p-m}{pm}} \Vert \nabla u_p\Vert_{L^{p}(\Omega)}.
		\end{equation}
		Thanks to \eqref{est2}, \eqref{est3} and the Morrey-Sobolev embedding from $W^{1,m}(\Omega)$ to H\"older spaces,
		\begin{equation}\label{est4}
			\vert u_p(x) - u_p(y)\vert \leq C_7 \vert x-y\vert^{1 - \alpha}
		\end{equation}
		with $\alpha = \frac{N}{m}$. 
		\vspace{0,2cm}
		
		Now, let us prove \ref{item2}. We consider as before $v(x) = \min_{y\in\partial\Omega} \psi(y) + d_H (y,x)$. We have
		\[
		\int_{\partial\Omega} (u_p - v) \dd (\Theta_p\cdot\mathbf{n})= \int_\Omega \Theta_p\cdot \nabla(u_p - v) \dd x - 	\int_\Omega (u_p - v) \rho \dd x.
		\]
		In other words
		\[
		\int_\Omega (u_p - v) \rho \dd x   =  \int_\Omega \Theta_p\cdot \nabla(u_p - v) \dd x  + \int_{\{u_p = \psi\}} (\psi - v) \dd(\Theta_p\cdot\mathbf{n})^{-} - \int_{\{u_p = \phi\}} (\phi - v) \dd (\Theta_p\cdot\mathbf{n})^{+}.
		\]
		We see that $\phi < v \leq \psi$ on $\partial\Omega$ so that $\psi - v \geq 0$ and $\phi - v <0$, thus $\phi - v < -C_1$ for some positive constant $C_1$. So we obtain
		
		\begin{equation}
			\label{eqs3}
			\int_{\Omega}\Theta_p\cdot\nabla u_p \dd x + C_1 \int_{\partial\Omega} \dd(\Theta_p\cdot\mathbf{n})^{+} \leq \int_\Omega (u_p - v) \rho \dd x  + \int_{\Omega} \Theta_p\cdot \nabla v \dd x.
		\end{equation}
		Since $H ^*$ is a Finsler metric, we have by Euler's homogeneous function theorem (see \textit{e.g.} \cite{nesterov2013introductory})  that $\partial_\xi H ^{*}(x,\xi)\cdot \xi =  H  ^*(x,\xi)$ for any $\xi\in\R^N$. Thus 
		\[
		\int_{\Omega}\Theta_p\cdot\nabla u_p \dd x = 	\int_{\Omega} H ^*(x,\nabla u_p)^{p-1} \partial_{\xi} H ^{*}(x,\nabla u_p)\cdot\nabla u_p\dd x   = \int_{\Omega} H ^*(x,\nabla u_p)^{p} \dd x .
		\]
		Using this fact in \eqref{eqs3}, we get
		\[
		\int_{\Omega} H ^*(x,\nabla u_p)^{p} \dd x + C _1\int_{\partial\Omega} \dd(\Theta_p\cdot\mathbf{n})^{+} \leq C_2 + \int_{\Omega} \Theta_p\cdot \nabla v \dd x,
		\]
		where $C_2 >0$ is independent from $p$.
		On the other hand, thanks to \eqref{cs} we have
		\[
		\begin{aligned}
			\int_{\Omega} \Theta_p\cdot \nabla v \dd x &\leq \int_{\Om} H(x,\Theta_p)H^{*}(x,\nabla v) \dd x\\
			&= \int_{\Om} H(x,H ^{*}(x,\nabla u_p) ^{p-1} \partial_{\xi}H ^{*}(x,\nabla u_p))H^{*}(x,\nabla v)  \dd x\\
			&=\int_{\Om} H ^{*}(x,\nabla u_p) ^{p-1} H(x,\partial_{\xi}H ^{*}(x,\nabla u_p))H^{*}(x,\nabla v)  \dd x\\
			&=\int_{\Om} H ^{*}(x,\nabla u_p) ^{p-1} H^{*}(x,\nabla v) \dd x, 
		\end{aligned}
		\]
		where we have used the homogeneity of $H$ and \eqref{fin_p3}. Using Hölder and Young's inequalities and the fact that $H^{*}(x,\nabla v) \leq 1$ a.e., we get
		\[
		\begin{aligned}
			\int_{\Om} H ^{*}(x,\nabla u_p) ^{p-1} H^{*}(x,\nabla v) \dd x &\leq \Big( 	\int_{\Om} H ^{*}(x,\nabla u_p) ^{(p-1)p^{'}}\dd x\Big)^{\frac{1}{p^{'}}}\vert\Om\vert^{\frac{1}{p}}\\
			&\le \frac{p-1}{p}\int_{\Om} H ^{*}(x,\nabla u_p) ^p \dd x +\frac{1}{p}\vert\Om\vert.
		\end{aligned}
		\]
		We deduce that
		\begin{equation}
			\frac{1}{p}\int_{\Om} H ^{*}(x,\nabla u_p) ^p \dd x + C_1 \int_{\partial\Omega} \dd(\Theta_p\cdot\mathbf{n})^{+} \leq C_2 +\frac{1}{p}\vert\Om\vert.
		\end{equation}
		Therefore 
		\begin{equation}
			\label{est_thetan1}
			\int_{\partial\Omega} \dd(\Theta_p\cdot\mathbf{n})^{+} \leq C_3
		\end{equation}
		for some positive constant $C_3$ independent of $p$. Set $w(x) = \max_{y\in\partial\Omega} \phi (y) - d_H ( y,x)$. Observe that $ \phi \leq w <\psi$ and following the same lines we get that 
		\begin{equation}
			\label{est_thetan2}
			\int_{\partial\Omega} \dd(\Theta_p\cdot\mathbf{n})^{-} \leq C_4.
		\end{equation}
		
		As for $\Theta_p$, we have
		\[
		\int_{\Omega} H ^*(x,\nabla u_p)^{p} \dd x   = \int_{\Omega}\Theta_p\cdot\nabla u_p \dd x = \int_{\partial\Om} u_p\dd(\Theta_p\cdot\mathbf{n}) + \int_{\Om} u_p\rho \dd x.
		\]
		Keeping in mind \eqref{est_thetan1} and \eqref{est_thetan2}, Hölder's inequality gives
		\begin{equation}\label{phi1}
			\int_{\Omega} H ^*(x,\nabla u_p)^{p-1}\dd x  \leq C_5,
		\end{equation}
		this proves \ref{item3}. 
	\end{proof}

Thanks to Proposition \ref{prop:main_estimates}, we can state the main result.
\begin{theorem}\label{thm:main_results}
	  Let $u_p$ be a minimizer of $\F_p$. Then, up to a subsequence, $u_p\rightrightarrows \uu$ on $\Ob,$ where  $\uu$ solves the following variant of Kantorovich-Rubinstein problem
	  \begin{equation}
	  	(\mathcal{KR})_H :~\max \Big\{ \int_{\Omega} u \dd \rho:~H ^{*}(x,\nabla u )\leq 1~\text{a.e.} ,~\phi\leq u \leq \psi~\mbox{on}~\partial\Omega\Big \}.
	  \end{equation}
	 Moreover, 	there exists a couple $(\Theta,\theta) \in\M_b(\Omega)^N\times \M_b(\partial\Omega),$   such that
		\begin{enumerate}[label=(\roman*)]
			\item    Up to a subsequence 
		$$(\Theta_p, \Theta_p\cdot\mathbf{n})  \rightharpoonup (\Theta,\theta) \quad \hbox{ in }\M_b(\Omega)^N\times \M_b(\partial\Omega)-\hbox{weak}^*.$$
		
		\item   $(\Theta,\theta)$   solves the Beckmann problem 
		 	\begin{equation}
			\B_H : \min_{\substack{\Phi\in\M_b(\Omega)^N \\ \nu\in\M_b(\partial\Omega)}}\left \{ \int_{\Omega} H (x,\frac{\Phi}{\vert\Phi\vert}) \dd\vert\Phi\vert + \int_{\partial\Omega} \psi \dd\nu^- - \int_{\partial\Omega} \phi \dd\nu^+ 
			:~-\operatorname{div}(\Phi) = \rho + \nu~\mbox{in}~\D ^{'}(\R^N)\right\}.
		\end{equation}
	
			\item\label{iii} The couple   $(\uu,\Theta)$   solves  the PDE 
		\begin{equation} \label{MKpde1}
		\left\{
		\begin{array}{ll}
			-\operatorname{div}(\Theta) =  \rho  \quad  & \mbox{ in }\Omega \\  \\ 
			\Theta(x)\cdot\nabla  \uu(x) = H\left(x, \Theta \right) & \mbox{ in }\Omega \\  \\ 
			\phi\leq \uu\leq \psi  & ~~\mbox{on}~~\partial\Om,
		\end{array}
		\right.
	\end{equation} 		
			in the following sense:  $(\uu,\Theta) \in \W_{\phi,\psi} \times    \M_b(\Omega)^N,$  $ \Theta    \cdot\mathbf{n}= \theta   \in  \M_b(\partial\Omega) ,$   
				\begin{equation} \label{HThetau}
		 \frac{\Theta}{\vert\Theta\vert}  \cdot\nabla_{\vert \Theta \vert } \uu  = H\left(.,\frac{\Theta}{\vert\Theta\vert} \right) ,\quad \vert \Theta\vert -\hbox{a.e. in } \Omega,   
	 		\end{equation} 
 			\begin{equation}  
 			\supp(\theta ^{+})\subset  \{\uu  = \phi\}\quad    \mbox{and}\quad  \supp(\theta^{-})\subset\{\uu  = \psi\},
 		\end{equation} 
 	and 
			\begin{equation}
				\label{mf}
				\int_{\Om} \Theta \cdot\nabla \eta\:  \dd x = \int_{\Om} \eta\rho\:  \dd x+ 	\int_{\partial\Om}\eta \:  \dd \theta ~~\mbox{for all}~\eta\in W^{1,\infty}(\Om). 
			\end{equation}	    			
			\end{enumerate}
\end{theorem}
			\begin{proof} \textbf{The case where the inequality \eqref{comp} is strict.}
				
		 First, we see that thanks to \eqref{estimate_up}, we have by Ascoli-Arzelà's theorem, up to a subsequence, $u_p\rightrightarrows \uu$ on $\Ob$ for some continuous function $\uu$ satisfying $\phi \leq \uu \leq \psi$ on $\partial\Om$.  It is clear that $\uu\in W^{1,\infty}(\Omega)$.
			
			We are now in a position to show that $\uu$ solves $\KR$. To do so, we take any $v\in \W_{\phi,\psi}$ such that $H^*(x,\nabla v (x))\le 1$ a.e.. Using the optimality of $u_p$ we see that
			\[
			-\int_{\Omega} u_p \rho\dd x\leq \F_p(u_p) \leq \F_p(v) \leq \frac{\vert\Omega\vert}{p} - \int_{\Omega} v \rho\dd x.
			\]
			Taking the limit up to a subsequence, we get
			\[
			\sup\Big\{\int_{\Omega} v \rho\dd x :~H ^*(x,\nabla v) \leq 1,~\mbox{a.e.},~\phi\leq v \leq \psi~\mbox{on}~\partial\Omega\Big\}\leq \int_{\Omega} \uu \rho \dd x.
			\]
			It remains to show that $\uu$ is $1-$Lipschitz with respect to $d_H $, that is, $H^*(x,\nabla \uu (x))\le 1$ a.e.. Recall that $\phi\leq \uu \leq \psi$ on $\partial\Omega$. Again, using \eqref{est1}, we consider $N<m\leq p$ and we use Hölder's inequality to get
			\[
			\Vert H ^*(x,\nabla u_p)\Vert_{L^m(\Omega)} \leq (C_5p) ^{\frac{1}{p-1}}\vert\Omega\vert^{\frac{p-m}{pm}}.
			\]
			Since $u_{p}\rightrightarrows \uu$ uniformly in $\Ob$,  we can assume that up to a subsequence $u_{p}\rightharpoonup\uu$ weakly in $W^{1,m}(\Omega)$, and particularly, $\nabla u_{p}\rightharpoonup \nabla\uu$ weakly in $L^{m}(\Omega,\R^N)$. Mazur's lemma (see \cite{ekeland1999convex} for example) ensures the existence of a convex combination of $\nabla u_{p_k}$ converging in norm toward $\nabla \uu$. More precisely, there exists $\{U_i\}$  such that
			\[
			U_i= \sum_{k=i}^{n_i}\alpha_{k}^{i}\nabla u_{p_k}
			\]
			where $\sum_{k=i}^{n_i}\alpha_{i}^{k} = 1$, and $\alpha_{k} ^{i}\geq 0,~i\leq k\leq n_i$ and $\Vert U_i - \nabla\uu \Vert_{L^m(\Omega)}\to 0$ as $i\to +\infty$.
			Since $H ^*$ is continuous, we have
			\begin{equation*}
				\begin{aligned}
					\Vert H ^*(x,\nabla\uu)\Vert_{L^m(\Omega)} &\leq  \liminf_{i\to\infty}\Vert H  ^{*}(x,\sum_{k=i}^{n_i}\alpha_{k}^{i}\nabla u_{p_k})\Vert_{L^m(\Omega)}\\
					&\leq  \liminf_{i\to\infty}\sum_{k=i}^{n_i}\alpha_{k}^{i}\Vert H ^{*}(x,\nabla u_{p_k})\Vert_{L^m(\Omega)}\\
					&\leq \liminf_{i\to\infty}\sum_{k=i}^{n_i}\alpha_{k}^{i} (C_5p_k) ^{\frac{1}{p_k-1}}\vert\Omega\vert^{\frac{p_k-m}{m p_k}} = |\Omega|^{\frac{1}{m}}.
				\end{aligned}
			\end{equation*}
Taking $m\to \infty$, we get $H^*(x,\nabla u(x))\leq 1,$ a.e. $x\in \Omega.$  On the other hand, we see that \eqref{estimate_Thetap} and \eqref{estimates_Thetapn} implies that 
		 $\Theta_p$  and $\Theta_p\cdot\mathbf{n}$ are bounded in $\M_b(\Ob)$ and $\M_b(\partial\Omega)$ respectively. As a consequence, there exists $\Theta\in \M_b(\Ob)^N$ and $\theta \in\M_b(\partial\Omega)$ such that up to a subsequence 
		 \[
		 \Theta_p\rightharpoonup \Theta~\mbox{weakly* as}~p\to \infty,	
		 \]
		 and
		 \[
		 \Theta_p\cdot\mathbf{n}\rightharpoonup \theta~\mbox{weakly* as}~p\to \infty. 
		 \]
	Next, take any admissible potential $v\in C^1(\Omega)$ for $\KR$ and an admissible couple of flows $(\Psi,\nu)\in \M_b(\Omega)^N\times \M_b(\partial\Omega) $ for $\B_H$. Since $H^*(x,\nabla v)\leq 1$ for $\text{a.e. } x\in \Omega$, we have
\begin{align*}
	\int_{\Omega} H(x,\frac{\Psi}{\vert\Psi\vert}) \dd\vert\Psi\vert &\geq \int_{\Omega} H(x,\frac{\Psi}{\vert\Psi\vert})  H^{*}(x,\nabla v)\dd\vert\Psi\vert \\
	&\geq  \int_{\Omega}\frac{\Psi}{\vert\Psi\vert} \nabla v \dd\vert\Psi\vert \\
	& \geq  \int_{\Omega} v \dd\rho + \int_{\partial\Omega} \phi \dd\nu^+ - \int_{\partial\Omega} \psi \dd\nu^-
\end{align*}
and consequently
\[
\int_{\Omega} H(x,\frac{\Psi}{\vert\Psi\vert}) \dd\vert\Psi\vert+ \int_{\partial\Omega} \psi \dd\nu^- - \int_{\partial\Omega} \phi \dd\nu^+\geq \int_{\Omega} v \dd\rho.
\]
In particular, this implies that  
\begin{equation}
\min\B_H\geq \max\KR.
\end{equation}   
On the other hand, using Hölder's inequality combined with \eqref{ehf}-\eqref{fin_p2}, we get
\begin{align}
	\int_\Omega H(x,\frac{\Theta}{\vert\Theta\vert}) \dd\vert\Theta\vert &\leq \liminf_p\int_\Omega H\Big(x,H ^{*}(x,\nabla u_p) ^{p-1} \partial_{\xi}H ^{*}(x,\nabla u_p)\Big) \dd x\\
	& = \liminf_p \int_\Omega H ^{*}(x,\nabla u_p) ^{p-1}  H(x,\partial_{\xi}H ^{*}(x,\nabla u_p)) \dd x\\
	&\leq \liminf_p\Big(\int_\Omega   H ^{*}(x,\nabla u_p) ^{p} \dd x\Big) ^{\frac{p-1}{p}}\\
	& = \liminf_p \Big(\int_\Omega   H ^{*}(x,\nabla u_p) ^{p-1}~\partial_{\xi}H ^{*}(x,\nabla u_p)\cdot\nabla u_p \dd x\Big) ^{\frac{p-1}{p}}\\
	& =  \liminf_p \Big(\int_\Omega  \nabla u_p \dd\Theta_p\Big) ^{\frac{p-1}{p}}\\
	& =  \liminf_p \Big(\int_\Omega  u_p \rho \dd x + \int_{\partial\Om}u_p\dd(\Theta_p\cdot\mathbf{n})\Big) ^{\frac{p-1}{p}}\\
	& =  \int_\Omega  \uu \rho \dd x +\int_{\partial\Om}\phi\dd\theta^{+} -  \int_{\partial\Om}\psi\dd\theta^{-}.
\end{align}
This implies that 
\begin{equation}
 \min\B_H\ \leq 	\int_\Omega H(x,\frac{\Theta}{\vert\Theta\vert}) \dd\vert\Theta\vert  -\int_{\partial\Om}\phi\dd\theta^{+}   +    \int_{\partial\Om}\psi\dd\theta^{-} \leq   \int_\Omega  \uu \rho \dd x =   \max\KR.
\end{equation}
Thus 
 \begin{equation}
	\min\B_H\ = 	\int_\Omega H(x,\frac{\Theta}{\vert\Theta\vert}) \dd\vert\Theta\vert  -\int_{\partial\Om}\phi\dd\theta^{+}   +    \int_{\partial\Om}\psi\dd\theta^{-} =   \int_\Omega  \uu \rho \dd x =   \max\KR,
\end{equation}
which implies the optimality  of $\bf{u}$ and $(\Phi, \theta)$.

\textbf{Now it remains to show the results for the general case where the inequality \eqref{comp} needs not to be strict.}

We proceed by approximations. Consider two sequences $\{\phi_n\}_n$ and $\{\psi_n\}_n$  of continuous functions on $\partial\Om$ such that
\[
\phi_n(x) - \psi_n(y) < d_{H}(y,x)~\mbox{for all}~x,y\in\partial\Om,
\]
and
\[
\phi_n\rightrightarrows \phi~~\mbox{and}~~\psi_n\rightrightarrows \psi~\mbox{on}~\partial\Om.
\]
Then, thanks to the previous case, there exists a sequence of $\{\uu_n\}_n \in \W_{\phi_n,\psi_n}$ such that $H^{*}(x,\nabla \uu_n)\leq 1~\mbox{a.e}~\Om$. In addition, consider the corresponding solutions to the Beckmann problem $(\Theta_n,\theta_n)$. We then have
\begin{equation}
	\label{opt_n}
\int_\Om \uu_n \dd\rho =  \int_{\Omega} H (x,\frac{\Theta_n}{\vert\Theta_n\vert}) \dd\vert\Theta_n\vert  -\int_{\partial\Om}\phi_n\dd\theta_{n}^{+}   +    \int_{\partial\Om}\psi_n\dd\theta_{n}^{-} = \min\B_H.
\end{equation}
Then we deduce by the previous arguments that
\[
\uu_n\rightrightarrows \uu~\mbox{uniformly in}~\Ob~\mbox{with}~~H^{*}(x,\nabla \uu)\leq 1~\mbox{a.e. and}~~\phi\leq \uu \leq \psi~\mbox{in}~\partial\Om. 
\]
Next, we follow the main ideas of the proof of Proposition \ref{prop:main_estimates}. Define $v_n(x) = \min_{y\in\partial\Om}\left\{\psi_n(y)+d_{H}(y,x)\right\}$. Then
\begin{equation}
	\label{large:ineq1}
\int_{\Om}\Theta_n\cdot\nabla \uu_n \dd x + C_1 \int_{\partial\Om} \dd \theta_{n}^{+} \leq \int_{\Om} (\uu_n - v_n)\rho\dd x + \int_{\Om}\Theta_n\cdot \nabla v_n \dd x,
\end{equation} 
where $C_1$ is a positive constant independent from $n$. Using \eqref{MKpde1}, we have
\[
\int_{\Om}\Theta_n\cdot\nabla \uu_n \dd x = \int_{\Om} H(x,\frac{\Theta_n}{\vert\Theta_n\vert}) \dd \vert\Theta_n\vert.
\] 
On the other hand, since $H^{*}(x,\nabla v_n(x)) \leq 1~\mbox{a.e}$, we get
 \[
\int_{\Om}\Theta_n\cdot\nabla v_n \dd x \leq \int_{\Om} H(x,\Theta_n)H^{*}(x,\nabla v_n) \dd x \leq \int_{\Om}H(x,\frac{\Theta_n}{\vert\Theta_n\vert}) \dd \vert \Theta_n \vert.
\]
Combining these facts in \eqref{large:ineq1}, and using \eqref{equiv} we get
\begin{equation}
	\label{large:est_theta1}
	 \int_{\partial\Om} \dd \theta_{n}^{+} \leq C, ~\mbox{with}~C>0.
\end{equation}
Similarly, working with $w_n(x) = \max_{y\in\partial\Omega} \phi_n(y) - d_H ( y,x)$ instead of $v_n$, we get
\begin{equation}
	\label{large:est_theta2}
	\int_{\partial\Om} \dd \theta_{n}^{-} \leq C, ~\mbox{with}~C>0.
\end{equation}
As for $\Theta_n$, we deduce from \eqref{equiv}, \eqref{opt_n}, \eqref{large:est_theta1} and \eqref{large:est_theta2} that
\[
 \int_{\Omega} \vert\Theta_n\vert \dd x \leq C.
\]
Then, up to a subsequence, $(\Theta_n,\theta_n)\rightharpoonup (\Theta,\theta)~\mbox{weakly* as}~n\to\infty$. Thus, passing to the limit in \eqref{opt_n}, the proof is complete.

Finally, for the proof of the last item \ref{iii}, by passing to the limit, we recover the conditions  			\begin{equation}
	\supp(\theta ^{+})\subset  \{\uu  = \phi\}\quad    \mbox{and}\quad  \supp(\theta^{-})\subset\{\uu  = \psi\},
	\end{equation} 
	and 
	\begin{equation}
	\int_{\Om} \Theta \cdot\nabla \eta\:  \dd x = \int_{\Om} \eta\rho\:  \dd x+ 	\int_{\partial\Om}\eta \:  \dd \theta ~~\mbox{for all}~\eta\in W^{1,\infty}(\Om).
	\end{equation}
The equation \begin{equation} 
\frac{\Theta}{\vert\Theta\vert}  \cdot\nabla_{\vert \Theta \vert } \uu  = H\left(.,\frac{\Theta}{\vert\Theta\vert} \right) ,\quad \vert \Theta\vert -\hbox{a.e. in } \Omega
\end{equation} 
is due to the optimality of $\bf{u}$ and $\Phi$ (see for example \cite{igbida2018optimal,nguyen2021monge}).

\end{proof}

	By uniqueness of the maximal viscosity subsolution of \eqref{hjobs} we easily deduce the following corollary.
	\begin{corollary}
		Let $H = \sigma$, with $\sigma$ being the support function of the $0$-sublevel sets of the Hamiltonian $F$ in \eqref{hjobs}. Then the whole sequence $\{\uu_p\}_p$ converges uniformly to the solution $\uu$ of \eqref{hjobs}. 
	\end{corollary}

Now let us state the PDE satisfied by the potential $\uu$ and the flow $\Theta$, which in particular will give a characterization of the HJ equation \eqref{hjobs}.

 \begin{proposition}
The couple $(\uu,\Theta)$ given by Theorem \ref{thm:main_results} is a solution of the following PDE 
 \begin{equation} 
 	\left\{
 	\begin{array}{ll}
 		-\operatorname{div}(\Theta) =  \rho  \quad  & \mbox{ in }\Omega \\  \\ 
 		\Theta \in \partial I\!\!I_{B_{H^{*}(x,.)}} (\nabla  \uu)  & \mbox{ in }\Omega \\  \\ 
 		\phi\leq \uu\leq \psi  & ~~\mbox{on}~~\partial\Om,
 	\end{array}
 	\right.
 \end{equation} 		
in the sense that:   $(\uu,\Theta) \in \W_{\phi,\psi} \times    \M_b(\Omega)^N, \Theta    \cdot\mathbf{n}= \theta   \in  \M_b(\partial\Omega) $,
 \begin{equation}   
 		\Theta \in \partial I\!\!I_{B_{H^{*}(x,.)}} (\nabla _{\vert \Theta\vert } \uu) 
 	 ,\quad \vert \Theta\vert -\hbox{a.e. in } \Omega, 
 \end{equation} 
 \begin{equation}  
 	\supp(\theta ^{+})\subset  \{\uu  = \phi\}\quad    \mbox{and}\quad  \supp(\theta^{-})\subset\{\uu  = \psi\},
 \end{equation} 
 and 
 \begin{equation} 
 	\int_{\Om} \Theta \cdot\nabla \eta\:  \dd x = \int_{\Om} \eta\rho\:  \dd x+ 	\int_{\partial\Om}\eta \:  \dd \theta ~~\mbox{for all}~\eta\in W^{1,\infty}(\Om). 
 \end{equation}
In particular, taking $H = \sigma$, with $\sigma$ being the support function of the $0$-sublevel sets of the Hamiltonian $F$, the maximal viscosity subsolution $\uu$ of \eqref{hjobs} is uniquely characterized by the existence of $\Theta \in \mathcal M_b(\Omega)^N$ such that the couple $(\uu,\Theta)$ is a solution of   the PDE  \begin{equation} 
	\left\{
	\begin{array}{ll}
		-\operatorname{div}(\Theta) =  1  \quad  & \mbox{ in }\Omega \\  \\ 
		\Theta \in \partial I\!\!I_{Z(x)} (\nabla  \uu)  & \mbox{ in }\Omega \\  \\ 
		\phi\leq \uu\leq \psi  & ~~\mbox{on}~~\partial\Om. 
	\end{array}
	\right.
\end{equation}  
  \end{proposition}
\begin{proof} 
The divergence and boundary constraints follow from Theorem \ref{thm:main_results} and 
$$ \Theta \in \partial I\!\!I_{B_{H^{*}(x,.)}} (\nabla _{\vert \Theta\vert } \uu)$$
 is recovered by  \eqref{HThetau}.  
\end{proof}

\bigskip 
For general $H,$ it is labyrinthine  to phrase the flow $\Theta$ explicitly in terms of the gradient of the potential $\uu$ and the transport density alike Evans-Gangbo like formula in \eqref{mksys}. The following result points out two  particular situations showing  how  this is possible.

\bigskip

\begin{corollary}\label{cor1}
	Let $(\uu,\Theta)$ be a solution of the PDE \eqref{MKpde1}  in the sense of  Theorem \ref{thm:main_results}. If 
	\begin{equation} \label{mh}
		\vert  \Theta\vert    \ll \mathcal{L}^N, \end{equation} 
	 then,  setting  
	 \begin{equation}
	 	\label{tdensity}
	 \omega:= H(x, \Theta),
\end{equation}
	we have 
	\begin{equation} 
		\Theta = \omega \: \partial_\xi H^*(x,\nabla \uu)\quad \mathcal{L}^N-\text{ a.e.  } x\in \Omega,
	\end{equation}
	and 
	\begin{equation}  
		\omega\: (H ^*(x,\nabla \uu ) -1)=0 \: \quad \mathcal{L}^N-\text{ a.e. } x\in \Omega.
	\end{equation} 
\end{corollary}
\begin{proof}   	If $\vert \Theta\vert \ll  \mathcal{L}^N,$ then  $\nabla_{\vert \Theta \vert  } \uu=\nabla \uu, \mathcal L^{N}-$ a.e.in $\Omega$, and by taking $\omega$ as in \eqref{tdensity}, the relationship  \eqref{HThetau} implies that $\Theta \cdot \nabla u= \omega~ \mathcal L^{N}-$ a.e. in $\Omega$.   Since, moreover $H^*(x,\nabla \uu)\leq 1,$ then by definition of $H^*,$ we get 
		\begin{equation}
		\Theta = \omega \: \partial_\xi H^*(x,\nabla_\omega \uu) \hbox{ and }\omega\: ( H ^*(.,\nabla_\omega \uu ) -1)= 0  , \quad \mathcal{L}^{N}-\hbox{a.e. in  }\Omega.
	\end{equation} 
\end{proof}

\begin{corollary}\label{mr}
	Let $(\uu,\Theta)$ be a solution of \eqref{MKpde1} in the sense of Theorem  \ref{thm:main_results}.  We set again 
	$$\omega :=    H(x, \Theta)$$ 
	and,  we   assume moreover that
	\begin{equation}
		\label{asp}
		H ^*(x,\nabla_\omega \uu ) \leq 1 \quad \omega-\text{a.e. } x\in \Omega.
	\end{equation}
	Then 
	\begin{equation}
\Theta = \omega \: \partial_\xi H^*(x,\nabla_\omega \uu),
	\end{equation}
and 
\begin{equation} \label{H*un}
	H ^*(x,\nabla_\omega \uu ) = 1\: \quad \omega-\text{a.e. } x\in \Omega.
\end{equation} 

\end{corollary}
\begin{proof}
	See that $\nabla_{\vert \Theta\vert }\uu=\nabla_{\omega} \uu$ and  
	\begin{equation} 	 H\left (x,\frac{\dd \Theta}{\dd\omega }\right)=1 \quad \omega  -\hbox{a.e. }\Omega. 
	\end{equation}  
So, in  one hand, using the fact that 
		\begin{equation}
		\nabla_{\vert \Theta\vert } u \cdot \frac{\Theta}{\vert \Theta\vert } =  H\left (x,\frac{\Theta}{\vert \Theta\vert }\right) \quad \vert \Theta\vert -\hbox{a.e. }\Omega. 
	\end{equation} 
	 we have 
	\begin{equation}
	 \nabla_\omega \uu\cdot \frac{\dd\Theta}{\dd\omega}   = 	\nabla_{\vert \Theta\vert } u \cdot \frac{\dd\Theta}{\dd\omega } = 1 \quad \omega-\hbox{a.e. }\Omega.
	\end{equation}
 On the other, we see that 
	 	\begin{equation}
	 \nabla_\omega u\cdot \frac{\dd\Theta}{\dd\omega}   \leq H^*(x, \nabla_\omega u)\:   H\left (x,\frac{\dd \Theta}{\dd\omega }\right)  =  H^*(x, \nabla_\omega u)   \quad \omega-\hbox{a.e. }\Omega.
	\end{equation}
So, assuming  \eqref{asp},  we get 
		\begin{equation}
		1= \nabla_\omega u\cdot \frac{\dd\Theta}{\dd\omega}   =  H^*(x, \nabla_\omega u)\:   H\left (x,\frac{\dd \Theta}{\dd\omega }\right)  =  H^*(x, \nabla_\omega u)   \quad \omega-\hbox{a.e. }\Omega.
	\end{equation}
 Thus the results follow by definition of $H^*.$

\end{proof}

\begin{remark} 

		Combining Theorem \ref{thm:main_results} and Corollaries  \ref{cor1}--\ref{mr}, the couple $(\omega:= H(x,\Theta), \uu)$ solves the  associated Monge-Kantorovich system to $\KR$ and $\B_H$:
	\begin{equation}
		\label{MK_H}
		\left\{\begin{array}{ll}-\operatorname{div}(\omega\partial_{\xi}H^{*}(x,\nabla_\omega \uu))=\rho & \text {in } \Omega \\ 
			\partial_{\xi}H^*(x,\nabla_\omega \uu)\cdot\mathbf{n} \geq 0 & \text {on }\left\{\uu  = \phi\right\} \\ 
			\partial_{\xi}H^*(x,\nabla _\omega \uu)\cdot\mathbf{n} \leq 0 & \text {on } \left\{\uu  =  \psi\right\} \\ 
			\partial_{\xi}H ^{*}(x,\nabla_\omega u)\cdot\mathbf{n}  = 0 & \text {in } \{\phi < \uu <\psi\}\\
			\phi \leq \uu \leq \psi & \text {on } \partial \Omega \\ 
			H ^{*}(x,\nabla_\omega \uu)\leq 1 & \text {in } \Omega \\ 
			H^{*}(x,\nabla _\omega \uu)= 1& \omega-\text {a.e.}\end{array}\right.
	\end{equation}
		In particular, given a positive continuous function $k:\Ob\to\R$, and define the following Finsler metric $H (x,p) = k(x)\vert p\vert$ for $(x,p)\in\Ob\times\R^N$. We easily see that its dual reads
		\[
		H ^{*}(x,q)  = \frac{\vert q\vert}{k(x)},
		\] 
		and the systems \eqref{MKp}-\eqref{MK_H} reduce the ones studied in \cite{dweik2018weighted}.
		
		 Moreover, if the Finsler metric is defined via the so called Minkowski functional (or gauge function) 
		\[
		\mathbf{g}_{K}(p) = \inf\{t>  0:~t^{-1} p \in K  \},
		\]
		where $K$ is a convex, closed and bounded set $\R^N$, then considering $H^*(x,p) = \mathbf{g}_{K}(p)$ and $\phi= \psi$, we recover the Monge--Kantorovich system studied in \cite{Crasta&Malusa}.

\end{remark}

	\section{Connection with Monge--Kantorovich problem}\label{connection}

	Let us recall that we can derive a dual problem to $\KR$ using perturbation techniques (as in \cite{dweik2018weighted,ennaji2020augmented}), to get the following Kantorovich problem
	
	\begin{equation}
		\KP:~ \min_{\gamma\in\Pi(\rho^+,\rho^-)} \Big\{\int_{\Ob\times\Ob} d_H(x,y) \dd\gamma(x,y) +  \int_{\partial\Omega} \psi(y) \dd (\pi_{y})_{\sharp}\gamma  - \int_{\partial\Omega} \phi(x) \dd (\pi_{x})_{\sharp}\gamma  \Big\}.
	\end{equation}
	Here $\Pi(\rho^+,\rho^-) = \{\gamma\in\M^{+}(\Ob\times\Ob):~(\pi_{x})_{\sharp}\gamma\mres\Omega = \rho ^+, (\pi_{y})_{\sharp}\gamma\mres\Omega = \rho ^-\}$, with $\pi_{x}$ and $\pi_{y}$ stand for the usual projections of $\Ob\times\Ob$ onto $\Ob$, that is $\pi_{x}(x,y) = x$ and $\pi_{y}(x,y) = y$ for any $(x,y)\in\Ob\times\Ob$ and
	\[
	(\pi_{x})_{\sharp}\gamma\mres\Omega = \rho ^+ \Leftrightarrow \gamma(A\times \Ob) = \rho^{+}(A)~\mbox{for any Borelean }A\subset \Om,
	\]
	\[
		(\pi_{y})_{\sharp}\gamma\mres\Omega = \rho ^- \Leftrightarrow \gamma(\Ob\times B) = \rho^{-}(B)~\mbox{for any Borelean }B\subset \Om.\]
		
	The existence of optimal solution to $\KP$ can be obtained using the direct method of calculus of variations. Moreover,  all the extremal values coincide:
	\begin{equation}
		\label{equalities_H}
	\min\B_H = \min\KP = \max\KR.
	\end{equation}
	Here $\phi$ and $\psi$ play the role of import/export costs for the Kantorovich problem $\K$ as in \cite{dweik2018weighted,mazon2014optimal} for the Euclidean and Riemannian costs. In addition, we show that the measure $\theta$ constructed in Theorem \ref{thm:main_results} will add to the measure $\rho$ so that the potential $\uu$ will be a Kantorovich potential for the classical transport problem on $\Ob$ between $\mu:=\rho ^{+}\lebes\mres\Om + \theta^+$ and $\nu:=\rho ^{-}\lebes\mres\Om + \theta^-,$ that is

	\[
	\int_{\Ob}\uu \dd(\mu - \nu) = \min_{\gamma\in\Gamma(\mu,\nu)} \int_{\Ob\times\Ob} d_{H}(x,y) \dd\gamma(x,y),
	\]
	where $\Gamma(\mu, \nu) := \{\gamma\in\M^{+}(\Ob\times\Ob):~(\pi_{x})_{\sharp}\gamma = \mu, (\pi_{y})_{\sharp}\gamma = \nu\}$ denotes the set of transport plans from $\mu$ to $\nu$ on $\Ob$. 

	\begin{proposition}
		\label{prop4}
		Let $\uu$ be the limit of the family of Finsler $p$-Laplace problems constructed in Theorem \ref{thm:main_results}. Then $\uu$ is a Kantorovich potential for the classical optimal transport problem between $\rho ^{+}\lebes\mres\Om + \theta^+$ and $\rho ^{-}\lebes\mres\Om + \theta^- $. Moreover
		\[
		\int_{\Om} \uu \rho \dd x = \min\KP.
		\]
		
	\end{proposition}
	\begin{proof}
		In the definition of $\Theta_p\cdot\mathbf{n}$ in \eqref{dist}, we take as a test function $\eta = \uu$ to get
		\[
		\int_{\partial\Omega} \uu \dd (\Theta_p\cdot\mathbf{n}) = \int_{\Omega} \Theta_p\cdot\nabla\uu \dd x - \int_{\Omega} \uu\rho\dd x.
		\]
		Thanks to Theorem \ref{thm:main_results}, passing to the limit $p\to\infty$ (up to a subsequence) we get 
	
		\begin{equation}
			\label{r1}
			\lim\limits_{p\to \infty} \int_{\Omega} \Theta_p\cdot\nabla\uu \dd x = 
			\int_{\partial\Omega} \uu \dd\theta + \int_{\Omega} \uu\rho\dd x.
		\end{equation}	
		Since $\uu$ is $1-$Lipschitz with respect to $d_H$, we may find thanks to Lemma \ref{l3}, a sequence of smooth functions $w_\epsilon$ converging uniformly to $\uu$ and enjoying the property of being $1-$Lipschitz with respect to $d_H$. By definition of $\Theta_p\cdot\mathbf{n}$, we get
		\[
		\int_{\partial\Omega} (\uu  - w_\epsilon) \dd (\Theta_p\cdot\mathbf{n}) =  \int_{\Omega} \Theta_p\cdot (\nabla\uu - \nabla w_\epsilon) \dd x - \int_{\Omega} (\uu - w_\epsilon)\rho\dd x.
		\]
		Taking $p\to \infty$ (again, up to a subsequence) and keeping in mind \eqref{r1}, we get
		\begin{equation}
			\label{e2}
			\int_\Omega \uu \rho \dd x + \int_{\partial\Omega} \uu \dd\theta = \int_{\Omega} (\uu - w_\epsilon)\rho\dd x + \int_{\partial\Omega} (\uu  - w_\epsilon) \dd\theta + \int_{\Omega} \Theta\cdot\nabla w_\epsilon \dd x  = A_\epsilon + B_\epsilon,
		\end{equation}
		with $A_\epsilon = \int_{\Omega} (\uu - w_\epsilon)\rho\dd x + \int_{\partial\Omega} (\uu  - w_\epsilon) \dd\theta $ and $B_\epsilon = \int_{\Omega} \Theta\cdot\nabla w_\epsilon \dd x $. Since $w_\epsilon$ converges uniformly to $\uu$ on $\Ob$, we have that $A_\epsilon\to 0$ as $\epsilon\to 0$. We claim that 
		\[
		B_\epsilon \to \int_\Omega H(x,\frac{\Theta}{\vert\Theta\vert}) \dd\vert\Theta\vert
		\]
		as $\epsilon\to 0$. We first observe that
		\begin{align*}
			\int_\Omega \uu \rho\dd x &= \lim_{\epsilon\to 0} \int_\Omega w_\epsilon \rho\dd x\\
			&\le\lim_{\epsilon\to 0} \int_\Omega \nabla w_\epsilon \frac{\Theta}{\vert\Theta\vert} \dd\vert\Theta\vert  + \int_{\partial\Om}\psi\dd\theta^{-}  - \int_{\partial\Om}\phi\dd\theta^{+}
			\\
			&\leq \lim_{\epsilon\to 0} \int_\Omega H^*(x,\nabla w_\epsilon)H(x,\frac{\Theta}{\vert\Theta\vert}) \dd\vert\Theta\vert+ \int_{\partial\Om}\psi\dd\theta^{-}  - \int_{\partial\Om}\phi\dd\theta^{+}\\
			&\leq \int_\Omega H(x,\frac{\Theta}{\vert\Theta\vert}) \dd\vert\Theta\vert+ \int_{\partial\Om}\psi\dd\theta^{-}  - \int_{\partial\Om}\phi\dd\theta^{+}
		\end{align*}
		where  we have used Lemma \ref{l3} for the last inequality. 
		
		Again we proceed as in the proof of Theorem \ref{thm:main_results}: since $\Theta_p\rightharpoonup \Theta$, we have by Reshetnyak's lower semicontinuity theorem, we get 
		
		\begin{align}
			\int_\Omega H(x,\frac{\Theta}{\vert\Theta\vert}) \dd\vert\Theta\vert &\leq \liminf_p \int_\Omega H(x,\frac{\Theta_p}{\vert\Theta_p\vert}) \dd\vert\Theta_p\vert\\
			& = \liminf_p\int_\Omega H\Big(x, H ^{*}(x,\nabla u_p) ^{p-1} \partial_{\xi}H ^{*}(x,\nabla u_p)\Big) \dd x\\
			& = \liminf_p \int_\Omega H ^{*}(x,\nabla u_p) ^{p-1}  H(x,\partial_{\xi}H ^{*}(x,\nabla u_p)) \dd x \\
			&\leq \liminf_p\Big(\int_\Omega   H ^{*}(x,\nabla u_p) ^{p} \dd x\Big) ^{\frac{p-1}{p}}\\
			& = \liminf_p \Big(\int_\Omega   H ^{*}(x,\nabla u_p) ^{p-1}~\partial_{\xi}H ^{*}(x,\nabla u_p)\cdot\nabla u_p \dd x\Big) ^{\frac{p-1}{p}}\\
			& =  \liminf_p \Big(\int_\Omega  \nabla u_p \dd\Theta_p \Big) ^{\frac{p-1}{p}}\\
			& =  \int_\Omega  \uu \rho \dd x + \int_{\partial\Om} \uu\dd\theta\\
			& = \lim_{\epsilon\to 0} \int_{\Omega} w_\epsilon \rho\dd x+\int_{\partial\Om} w_\epsilon\dd\theta
		\end{align}
		where we have used  Hölder's inequality combined with \eqref{ehf} and \eqref{fin_p3}. Coming back to \eqref{e2} we get
		\begin{equation}
			\label{e3}
			\int_\Omega \uu \rho \dd x + \int_{\partial\Omega} \uu \dd\theta  = \int_\Omega H(x,\frac{\Theta}{\vert\Theta\vert}) \dd\vert\Theta\vert.
		\end{equation}
		To conclude, let us observe that taking $v\in W ^{1,\infty}(\Omega)$ such that $H ^{*}(x,\nabla v(x)) \leq 1$, we have
		\begin{align*}
			\int_\Omega \uu \rho \dd x + \int_{\partial\Omega} \uu \dd\theta  &= \int_\Omega H(x,\frac{\Theta}{\vert\Theta\vert}) \dd\vert\Theta\vert\\
			&\geq \int_\Omega  \frac{\Theta}{\vert\Theta\vert} \cdot \nabla v\dd\vert\Theta\vert\\
			& = \int_\Omega \nabla v \dd\Theta = \int_\Omega v\rho \dd x + \int_{\partial\Omega} v\dd\theta. 
		\end{align*}
		Thanks to \eqref{equalities_H} and the classical Kantorovich duality, we have 
		\[
		\int_\Omega \uu \rho \dd x + \int_{\partial\Omega} \uu \dd\theta  = \int_{\Ob\times\Ob} d_H(x,y)\dd\gamma(x,y),
		\]
		where $\gamma$ is an optimal plan of 
		\[
		\min  \Big\{\int_{\Ob\times\Ob} d_H (x,y) \dd\gamma(x,y):~(\pi_{x})_{\sharp}\gamma= \rho ^{+}\lebes\mres\Om + \theta^{+}, (\pi_{y})_{\sharp}\gamma= \rho ^{-}\lebes\mres\Om + \theta^{-} \Big\}.
		\]
		Since $(\pi_{x})_{\sharp}\gamma\mres\partial\Omega = \theta ^+$ and  $(\pi_{y})_{\sharp}\gamma\mres\partial\Omega = \theta ^-$ we deduce that
		\[
		\int_\Omega \uu \rho \dd x  = \int_{\Ob\times\Ob} d_H(x,y)d\gamma(x,y) + \int_{\partial\Omega} \psi\dd\theta^- - \int_{\partial\Omega} \phi\dd\theta^+ = \min\KP.
		\]
	\end{proof}

	\section{Appendix}
	Let us recall some facts concerning the notion of tangential gradient which played an important role in the previous proofs. To give a glimpse on the necessity to introduce this notion, let us remember that Beckmann's transportation problem is an optimisation problem on measure space under a divergence constraint. More particularly, the flow satisfies $-\operatorname{div}(\Phi) = \mu \in \M_b(\Ob)$. To do further analysis on such a problem and particularly to derive its dual problem we naturally attempt to integrate by parts in the divergence constraint and write, for some Lipschitz function $u$
	\[
	\int \nabla u\cdot\sigma~ \dd \gamma= \int u \dd\mu,
	\]
	where $\gamma = \vert\Phi\vert$ and $\sigma = \frac{\Phi}{\vert\Phi\vert}$. Observe that $\nabla u$ may not be well-defined on a $\vert\Phi\vert$-positive measure set and thus the previous formula may not have sense. Thanks to \cite{bouchitte1997energies} it is possible to give a sense to the previous formula using the notion of  tangential gradient as follows. First we can define the tangent space to the measure $\gamma$ 
	\[
	\mathcal{X}_\gamma (x) = \gamma-\text{ess}\cup \Big\{ \sigma(x):~ \sigma\in L^{1}_{\gamma}(\Ob,\R^N),~\operatorname{div}(\sigma\gamma)\in\M_b(\Ob)\Big\}.
	\]
	Then, the tangential gradient $\nabla_\gamma u(x)$ to a function $u\in C^1(\Ob)$ at $x$ with respect to the measure $\gamma$ is the orthogonal projection of $\nabla u(x)$ onto $\mathcal{X}_{\gamma}(x)$. Denoting by $\textbf{P}_{\gamma}(x)$ the orthogonal projection on $\mathcal{X}_\gamma (x)$, it has been shown in \cite{bouchitte2005completion} that the linear operator $u\in C ^1(\Ob) \to \nabla_\gamma u(x):=\textbf{P}_{\gamma}(x)\nabla u(x)\in L^{\infty}_{\gamma}(\Ob,\R^N)$ can be uniquely extended to a linear continuous operator
	\[
	\nabla_\gamma: u\in \hbox{Lip}(\Ob)\to \nabla_\gamma u\in L^{\infty}_{\gamma}(\Ob,\R^N).
	\]
Moreover, we have the following useful integration by parts formula
\begin{proposition}[\cite{bouchitte2005completion}]
	\label{IPP_TG}
	Given $\gamma\in\M_{b}^{+}(\Ob)$ and $\upsilon\in L^{1}_{\gamma}(\Ob,\R^N)$ such that $\upsilon(x)\in\mathcal{X}_{\gamma}(x)$ for $\gamma-$\text{a.e} $x$. and $\operatorname{div}(\gamma\upsilon):= \rho\in\M_{b}(\Ob)$. One then has
	\[
	\int_{\overline \Omega} u \dd\rho= \int_{\overline\Omega}\upsilon \nabla_\gamma u \dd \gamma,
	\]
	for any $u\in\mathrm{Lip}(\Ob)$.
\end{proposition}
	To end this section let us recall the following useful approximation result  \cite[Lemma A.1]{igbida2018augmented} (see also \cite[Lemma 3.1]{nguyen2021monge} for degenerate case of $H$).
	\begin{lemma}\label{l3} Let $H$ be a non-degenerate Finsler metric
 and $u\in W ^{1,\infty}(\Omega)$ such that $H ^{*}(x,\nabla u(x)) \leq 1$ for \text{ a.e.} $x\in \Omega$. Then, there exists a sequence of $u_\epsilon\in C^{1}(\Ob)$ such that $u_\epsilon \rightrightarrows u$ uniformly on $\overline\Omega$  as $\epsilon\to 0$  and
		\[
		H^* (x,\nabla u_\epsilon (x)) \leq 1 \text{ for all } x\in \overline \Omega. 
		\]
	\end{lemma}

\bibliographystyle{abbrv}
\bibliography{references_all}

\end{document}